\documentclass[reqno,12pt]{amsart}
\usepackage{a4wide}
\usepackage{amsmath}
\usepackage{amssymb}
\usepackage{amsthm}

\numberwithin{equation}{section}

\newtheorem{theo}{Theorem}

\newtheorem{lem}{Lemma}

\theoremstyle{remark}

\def\al{\alpha}
\def\be{\beta}

\def\de{\delta}
\def\ep{\varepsilon}

\def\({\left(}
\def\){\right)}
\def\[{\left[}
\def\]{\right]}
\def\ii{\infty}

\def\lcm{\operatorname{lcm}}

\def\dd{\textup{d}}

\def\cl#1{\lceil#1\rceil}

\begin{document}

\title{On a linear form for Catalan's constant}

\author{C. Krattenthaler$^\dagger$ and T. Rivoal}
\date{}

\address{C. Krattenthaler, Fakult\"at f\"ur Mathematik, Universit\"at Wien,
Nordbergstra{\ss}e~15, A-1090 Vienna, Austria.
WWW: \tt http://www.mat.univie.ac.at/\~{}kratt.}

\thanks{$^\dagger$Research partially supported 
by the Austrian Science Foundation FWF, grant S9607-N13,
in the framework of the National Research Network
``Analytic Combinatorics and Probabilistic Number Theory"}

\address{T. Rivoal,
Institut Fourier,
CNRS UMR 5582, Universit{\'e} Grenoble 1,
100 rue des Maths, BP~74,
38402 Saint-Martin d'H{\`e}res cedex,
France.\newline
WWW: \tt http://www-fourier.ujf-grenoble.fr/\~{}rivoal.}

\dedicatory{Dedicated to George Andrews}

\subjclass[2000]{Primary 11J72;
 Secondary 11J82, 33C20}

\keywords{Catalan's constant, linear forms, hypergeometric series,
Andrews' identity}

\begin{abstract}
It is shown how Andrews' multidimensional extension
of Watson's transformation between
a very-well-poised $_8\phi_7$-series and a balanced $_4\phi_3$-series
can be used to give a straightforward proof of a conjecture of Zudilin
and the second author on the arithmetic behaviour of the coefficients of
certain linear forms of $1$ and Catalan's constant. This proof is
considerably simpler and more stream-lined than the first proof,
due to the second author. 
\end{abstract}

\maketitle

\section{Introduction}

Andrews' multidimensional extension
\cite[Theorem~4]{AndrAE} of Watson's transformation between
a very-well-poised $_8\phi_7$-series and a balanced $_4\phi_3$-series
\cite[(2.5.1); Appendix (III.18)]{GaRaAA} in its full beauty reads
\begin{multline} \label{eq:Andrq}
\sum _{k=0} ^{n}
\frac {
(a;q)_k\,(q\sqrt a;q)_k\,(-q\sqrt a;q)_k\,( b_1;q)_k\,(  c_1;q)_k \cdots (b_{m+1};q)_k\, 
( c_{m+1};q)_k\,( q^{-n} ;q)_k}
  {(   \sqrt a;q)_k\,(-\sqrt a;q)_k\,( qa/b_1;q)_k\,( qa/c_1;q)_k \cdots 
  ( qa/b_{m+1};q)_k\,( qa/c_{m+1};q)_k\,(q^{n+1}a;q)_k} \\
\cdot
\left(\frac {a^{m+1}q^{m+1+n}} {b_1c_1\cdots
b_{m+1}c_{m+1}}\right)^k\\
=
\frac{(qa;q)_n\,(qa/b_{m+1}c_{m+1};q)_n}{(qa/b_{m+1};q)_n\,(qa/c_{m+1};q)_n}
\sum _{0\le i_1\le i_2\le\dots\le i_{m}\le n} ^{}
\frac {a^{i_1+\dots+i_{m-1}}q^{i_1+\dots+i_m}} {(b_2c_2)^{i_1}\cdots
(b_mc_m)^{i_{m-1}}}\kern2cm\\
\cdot\frac{(q^{-n};q)_{i_m}}{(b_{m+1}c_{m+1}/aq^n;q)_{i_m}}
\prod_{k=1}^m
\frac{(qa/b_{k}c_k;q)_{i_k-i_{k-1}}\,(b_{k+1};q)_{i_k}\,(c_{k+1};q)_{i_k}}
{(q;q)_{i_k-i_{k-1}}\,(qa/b_k;q)_{i_k}\,(qa/c_k;q)_{i_k}},
\end{multline}
where, by definition, $i_0:=0$.
Here, $(\al;q)_k=(1-\al)(1-\al q)\cdots(1-\al q^{k-1})$ if $k\ge1$
and $(\al;q)_0=1$.
This formula has found important applications to the theory of
partition identities (see \cite{AndrAE}).

Remarkably, Andrews' formula has started a surprising new life
recently. Its utility for proving arithmetic properties of
coefficients of certain linear forms for values of the Riemann zeta
function at integers was discovered by the authors in \cite{KrRiAA}, 
and was also exploited in \cite{KrRiAB} for proving the equality of
certain multiple integrals and hypergeometric series. Closely related
are the applications given by Zudilin in \cite{ZudiAN,ZudiAJ}. 
The afore-mentioned articles make actually ``only" use of the $q=1$
special case of \eqref{eq:Andrq} (see \eqref{eq:Andr} below for the
explicit statement of that special case).
The line of argument developed in \cite{KrRiAA} has been extended to
the $q$-case by Jouhet and Mosaki in \cite{JoMoAA} to establish 
irrationality results for values of a $q$-analogue of the zeta
function. Moreover, Guo, Jouhet and Zeng \cite{GuJZAA} have extended Zudilin's work in
\cite{ZudiAN} to the $q$-case, together with further applications of
Andrews' formula \eqref{eq:Andrq}. In a completely different field,
Beliakova, B\"uhler and L\^e \cite{BeBuAA,BeLeAA,LeTQAA} have
exploited \eqref{eq:Andrq} in the study of quantum invariants of
manifolds. Finally, Andrews himself returned to his identity after
over 30 years to prove deep partition theorems in \cite{AndrCM}.

The purpose of the present paper is to add another item to this list of
applications of Andrews' formula. More precisely, we show how the
ideas from \cite{KrRiAA} lead to an alternative proof of a conjecture 
from \cite{RiZuAA} on the arithmetic behaviour of the coefficients in
certain linear forms of $1$ and {\it Catalan's constant\/}
$G=\sum _{k=1} ^{\infty}\frac {(-1)^{k-1}} {(2k-1)^2}$. 
It is considerably simpler and more stream-lined than the first proof
\cite{RivoAG} by one of the authors, which used a somewhat indirect method
based on Pad\'e approximations. A partial, ``asymptotic," proof had been given
earlier by Zudilin in \cite{zu6}.

We give a precise statement of the conjecture in the next section,
where we also derive explicit expressions for the coefficients $\mathbf a_n$ and
$\mathbf b_n$ in the linear forms of $1$ and Catalan's constant. The
arithmetic claim for the coefficient $\mathbf a_n$ is then proved in
Section~\ref{sec:3} with the help of a limit case of Whipple's
transformation between a very-well-poised $_7F_6$-series and a
balanced $_4F_3$-series (the latter being the $q=1$ special case of the
afore-mentioned transformation formula of Watson). The
arithmetic claim for the coefficient $\mathbf b_n$ is proved in
Section~\ref{sec:4} with the help of the $q=1$ special case of
Andrews' formula \eqref{eq:Andrq}, given explicitly in
\eqref{eq:Andr}. 

\section{A linear form for Catalan's constant}
\label{sec:2}

Let us consider the series
\begin{equation} \label{eq:linG}
{\mathbf G}_n=n!\sum_{k=1}^{\infty}(-1)^k\bigg(k+\frac{n-1}{2} \bigg) \,
\frac{(k-n)_n(k+n)_n}{\left(k-\frac 12\right)_{n+1}^3},
\end{equation}
where the {\it Pochhammer symbol\/} $(\al)_k$ is defined by
$(\al)_k=\alpha(\alpha+1)\cdots(\alpha+k-1)$ if $k\ge 1$ and
$(\al)_0=1$.
By applying a partial fraction decomposition with respect to $k$ to
the summand, and by performing the appropriate summations, it is not
difficult to see (cf.\ \cite[Sec.~1.4]{fi} for details on this kind of
calculation) that
$$
{\mathbf G}_n={\bf a}_nG-{\bf b}_n,
$$
where
\begin{multline} \label{eq:an1}
{\mathbf a}_n=4(-1)^{n-1}\sum _{j=0} ^{n}\frac {\partial} {\partial \ep}
\Bigg(\(\frac {n} {2}-j+\ep\)
\(\frac {n!} { (1-\ep)_j \, (1+\ep)_{n-j}}\)^3\\
\cdot
\binom {n+j-\ep-\frac {1} {2}}n\binom
{2n-j+\ep-\frac {1} {2}}n\Bigg)\Bigg\vert_{\ep=0},
\end{multline}
and
\begin{multline} \label{eq:bn1}
{\mathbf b}_n=(-1)^n\sum _{j=0} ^{n}\sum _{e=1} ^{3}
\frac {1} {(3-e)!}\frac {\partial^{3-e}} {\partial j^{3-e}}
\Bigg(\(\frac {n} {2}-j+\ep\)
\(\frac {n!} { (1-\ep)_j \, (1+\ep)_{n-j}}\)^3\\
\cdot
\binom {n+j-\ep-\frac {1} {2}}n\binom
{2n-j+\ep-\frac {1} {2}}n\Bigg)\Bigg\vert_{\ep=0}
\sum _{k=1} ^{j}\frac {(-1)^{k}}
{\(k-\frac {1} {2}\)^e}.
\end{multline}
Writing $\textup d_n$ for $\lcm(1,2,\dots,n)$,
it is easy to see by a standard approach (see \cite[Sec.~5]{RiZuAA}) that 
$2^{4n}\textup d_{2n}\mathbf a_n$ and
$2^{4n}\textup d_{2n}^3\mathbf b_n$ are integers. 
Based on computer calculations, the second author and Zudilin
conjectured however (cf.\ \cite[p.~720]{RiZuAA}) that in fact even
$2^{4n}\mathbf a_n$ and
$2^{4n}\textup d_{2n}^2\mathbf b_n$ are integers. 
While this is still too weak for proving the irrationality of
Catalan's constant $G$, it is nevertheless an interesting and 
non-obvious observation
which we shall prove in the two subsequent sections. This proof
makes use of identities for ({\it generalised\/}) {\it hypergeometric
series}, the latter being defined by
\begin{equation*}
{}_{q+1}F_q\!\left [ \begin{matrix}
{ \alpha_0,\alpha_1,\ldots,\alpha_{q}}\\
{\beta_1,\ldots,\beta_q}\end{matrix} ;
{\displaystyle z}\right ]
=\sum_{k=0}^{\ii}
\frac{(\alpha_0)_k\,(\alpha_1)_k\cdots(\alpha_{q})_k}
{k!\,(\beta_1)_k\cdots(\beta_q)_k} z^k.
\label{eq:hyper}
\end{equation*}
As we already
mentioned in the Introduction, an earlier (but more involved) proof
is due to one of the authors \cite{RivoAG}.

\section{The coefficient ${\mathbf a}_n$}
\label{sec:3}

The purpose of this section is to prove the following theorem.

\begin{theo} \label{thm:an}
For all positive integers $n$, the number $2^{4n}\mathbf a_n$ is an
integer.
\end{theo}

For accomplishing the proof of this theorem (as well as the proof of
Theorem~\ref{thm:bn} in the following section), we need the
following two arithmetic auxiliary facts (cf.\ \cite[Sec.~7]{zu3}
and \cite[Lemma~6]{RiZuAA}, respectively). Following \cite{zu3} (where
this is attributed to Nesterenko), we shall call the expressions
$R_1(\al,\be;t)$ and $R_2(\al,\be;t)$ in the two lemmas below
{\it elementary bricks}.

\begin{lem} \label{lem:R1}
Given integers $\al$ and $\be$, let
$$R_1(\al,\be;t)=
\begin{cases} \dfrac {(t+\be)_{\al-\be}} {(\al-\be)!}
&\text{if }\al\ge \be,\\
\dfrac {(\be-\al-1)!} {(t+\al)_{\be-\al}}&\text{if }\al<\be.\end{cases}
$$
Then, for all integers $\al,\be,k,H$ with $\al\ge\be$ and $H\ge0$, the
number
\begin{equation*} 
\textup{d}_{\al-\be}^H \cdot\frac {1} {H!}\,
\frac {\partial ^H} {\partial t^H}R_1(\al,\be;t)\Big\vert_{t=-k}
\end{equation*}
is an integer. Furthermore, 
for all integers $\al,\be,k,H$ with $\al\le k\le \be-1$ and $H\ge0$, the
number
\begin{equation*} 
\textup{d}_{\be-\al-1}^H \cdot\frac {1} {H!}\,
\frac {\partial ^H} {\partial t^H}R_1(\al,\be;t)(t+k)\Big\vert_{t=-k}
\end{equation*}
is an integer.
\end{lem}

\begin{lem} \label{lem:R2}
Given integers $\al$ and $\be$ with $\al\ge\be$, let
$$R_2(\al,\be;t)=2^{2(\al-\be)}
\dfrac {(t+\be-\frac {1} {2})_{\al-\be}} {(\al-\be)!}.
$$
Then, for all integers $k$ and $H$ with $H\ge0$, the
number
\begin{equation*} 
\textup{d}_{2(\al-\be)}^H \cdot\frac {1} {H!}\,
\frac {\partial ^H} {\partial t^H}
R_2(\al,\be;t)\Big\vert_{t=-k}
\end{equation*}
is an integer.
\end{lem}

In order to apply these two lemmas, we need an alternative expression
for the coefficient $\mathbf a_n$, see the lemma below. The
expression in \eqref{eq:an2} was already given in
\cite[Sec.~4.1]{RivoAG}. Again, it was obtained there in a somewhat
roundabout way. Here, the equality in the next-to-last displayed equation in
\cite[Sec.~4.2]{RivoAG} is explained directly.

\begin{lem} \label{lem:an}
For all non-negative integers $n$, we have
\begin{equation} \label{eq:an2}
\mathbf a_n=-4\sum _{j=0} ^{n}\binom nj
\binom {n-\frac {1} {2}}j \binom {n+j-\frac {1} {2}}j.
\end{equation}
\end{lem}
\begin{proof} 
We loosely follow analogous considerations in \cite[Lemme~7]{KrRiAA}.

Let $H_m$ denote the $m$-th
harmonic number, defined by $H_m=
\sum _{j=1} ^{m}\frac {1} {j}$. By abuse of notation, we ``extend"
harmonic numbers to half-integers $m$ by defining $H_m=
\sum _{j=1} ^{\cl{m}}\frac {1} {m-j+1}$. For example, 
$$H_{5/2}=\frac {1} {5/2}+\frac {1} {3/2}+\frac {1} {1/2}.$$

We rewrite the expression for $\mathbf a_n$ given in
\eqref{eq:an1} in the form
\begin{align*} 
\mathbf a_n&=4(-1)^{n-1}\sum _{j=0} ^{n}
\(\frac {n} {2}-j\){\binom nj}^3
\binom {n+j-\frac {1} {2}}n\binom
{2n-j-\frac {1} {2}}n\\
&\kern3cm
\cdot
\(\frac {1} {\frac {n} {2}-j}+3H_j-3H_{n-j}+H_{2n-j-\frac
{1} {2}}-H_{n+j-\frac {1} {2}}-H_{n-j-\frac {1} {2}}+H_{j-\frac {1}
{2}}\)
\\
&=4(-1)^{n-1}\lim_{\ep\to0}\frac {2} {\ep}\sum _{j=0} ^{\infty} \left( {\frac n 2} +{\frac {\ep} 2} -j \right)
{\binom n j}\\
&\kern4cm
\cdot
\frac {({ \textstyle n - j - \ep +1}) _{j}} {({ \textstyle 1 - \ep}) _{j}}
\frac {({
  \textstyle n - j + \ep + 1}) _{j}} {({ \textstyle 1 - 2 \ep}) _{j}}
\frac {({ \textstyle j +{\frac 1 2}}) _{n}} {({ \textstyle 1 - 
\ep}) _{n} }
\frac {
  ({ \textstyle n - j +\ep+\frac {1} {2}}) _{n}  } 
{({ \textstyle 1 + \ep}) _{n} }.
\end{align*}
In hypergeometric notation, this reads
\begin{multline*} 
\mathbf a_n=4(-1)^{n-1}\lim_{\ep\to0}\frac {\left( n + \ep \right)  
     ({ \textstyle {\frac 1 2}}) _{n} \,
     ({ \textstyle  n + \ep + {\frac 1 2}}) _{n} } 
   {\ep ({ \textstyle 1 - \ep}) _{n} \,({ \textstyle 1 + \ep}) _{n}
}\\
\times
  {} _{6} F _{5} \!\left [ \begin{matrix} {  - n -\ep, 1 - {\frac n 2}
   - {\frac {\ep} 2}, n + {\frac 1 2}, -n,  - n + \ep, {\frac 1 2} - n
   - \ep}\\ { - {\frac n 2} -{\frac {\ep} 2}, {\frac 1 2} - 2 n - \ep,
      1 - \ep, 1 - 2 \ep, {\frac 1 2}}\end{matrix} ; {\displaystyle
      -1}\right ].
\end{multline*}
To the $_6F_5$-series we apply
the transformation formula (see \cite[(3.10.4), $q\to 1$]{GaRaAA})
\begin{multline} \label{eq:6F5}
{} _{6} F _{5} \!\left [ \begin{matrix} { a, 1 + {\frac a 2}, b, x, y, -N}\\ {
{\frac a
   2}, 1 + a - b, 1 + a - x, 1 + a - y, 1 + a + N}\end{matrix} ; {\displaystyle
   -1}\right ]  \\
=   \frac {( \textstyle 1 + a)_N\,( 1 + a - x - y) _{N}} 
   {( \textstyle 1 + a - x)_N\,
    1 + a - y) _{N}} 
 {} _{3} F _{2} \!\left [ \begin{matrix} { -N, x, y}\\ { -a - N + x
    + y, 1 + a - b}\end{matrix} ; {\displaystyle 1}\right ] ,
\end{multline}
where $N$ is a non-negative integer. Thus, we obtain
\begin{align*}
\mathbf a_n&=4(-1)^{n-1}\lim_{\ep\to0}\frac {{{\left( -1 \right) }^n} 
      \,n! } { ({ \textstyle 1 - \ep}) _{n} }
{} _{3} F _{2} \!\left [ \begin{matrix} { -n, n + {\frac 1 2}, {\frac 1 2} -
 n - \ep}\\ { 1, 1 - 2 \ep}\end{matrix} ; {\displaystyle 1}\right ]\\
&=-4\sum _{j=0} ^{n}\binom nj
\binom {n-\frac {1} {2}}j \binom {n+j-\frac {1} {2}}j,
\end{align*}
as we claimed.
\end{proof}

We are now in the position to prove Theorem~\ref{thm:an}.
\begin{proof}[Proof of Theorem~\ref{thm:an}] 
By Lemma~\ref{lem:R2} with $\al=j$, $\be=H=0$, and $k=-n$
respectively $k=-n-j$, the numbers $2^{2n}\binom {n-\frac {1} {2}}j$
and $2^{2n}\binom {n+j-\frac {1} {2}}j$ are integers.
Given the expression for $\mathbf a_n$ in Lemma~\ref{lem:an}, this
implies the assertion of the theorem.
\end{proof}

\section{The coefficient ${\mathbf b}_n$}
\label{sec:4}

The purpose of this section is to prove the following theorem.

\begin{theo} \label{thm:bn}
For all positive integers $n$, the number 
$2^{4n}\textup{d}_{2n}^2\mathbf b_n$ is an integer.
\end{theo}

\begin{proof}
This proof follows loosely analogous considerations in 
\cite[Prop.~7]{KrRiAA}. It depends on an arithmetic fact which is stated
and proved separately in Lemma~\ref{lem:briques1} below.

Let us start by reordering the summations in \eqref{eq:bn1} to obtain
\begin{multline} \label{eq:bn2}
{\mathbf b}_n=(-1)^n\sum _{e=1} ^{3}\sum _{k=1} ^{n}\frac {(-1)^{k}}
{\(k-\frac {1} {2}\)^e}
\frac {1} {(3-e)!}\frac {\partial^{3-e}} {\partial \ep^{3-e}}
\Bigg(\sum _{j=k} ^{n}\(\frac {n} {2}-j+\ep\)
\(\frac {n!} { (1-\ep)_j \, (1+\ep)_{n-j}}\)^3\\
\cdot
\binom {n+j-\ep-\frac {1} {2}}n\binom
{2n-j+\ep-\frac {1} {2}}n\Bigg)\Bigg\vert_{\ep=0}.
\end{multline}

Also for $\mathbf b_n$, we need an alternative expression.
It is provided for by the $q=1$ special case of Andrews' identity
\eqref{eq:Andrq}. More precisely, in \eqref{eq:Andrq} on replaces $a$
by $q^a$, $b_i$ by $q^{b_i}$, $c_i$ by $q^{c_i}$, $i=1,2,\dots,m+1$,
and then lets $q$ tend to $1$. As a result, one obtains the
transformation formula
\begin{multline} \label{eq:Andr}
{}_{2m+5}F_{2m+4}\left[
\begin{array}{c}
a,\frac{a}{2}+1, b_1,  c_1, \ldots,  b_{m+1},  c_{m+1}, -n \\
     \frac{a}{2}, 1+a-b_1, 1+a-c_1, \ldots,  1+a-b_{m+1}, 1+a-c_{m+1}, 1+a+n
\end{array}
\,;\, 1 \right]
\\
=
\frac{(1+a)_n\,(1+a-b_{m+1}-c_{m+1})_n}{(1+a-b_{m+1})_n\,(1+a-c_{m+1})_n}
\sum _{0\le i_1\le i_2\le\dots\le i_{m}\le n} ^{}
\frac{(-n)_{i_m}}{(b_{m+1}+c_{m+1}-a-n)_{i_m}}\kern2cm\\
\cdot
\Bigg(\prod_{k=1}^m
\frac{(1+a-b_{k}-c_k)_{i_k-i_{k-1}}\,(b_{k+1})_{i_k}\,(c_{k+1})_{i_k}}
{(i_k-i_{k-1})!\,(1+a-b_k)_{i_k}\,(1+a-c_k)_{i_k}}\Bigg),
\end{multline}
where again, by definition, $i_0:=0$.
In this formula we put $m=3$, $a=-n+2k-2\ep$, $b_1=-n+k-\ep$,
$b_2=-n+k-\ep+\frac {1} {2}$, $c_2=n+k-\ep+\frac {1} {2}$,
$b_3=-n+k -\ep$, $c_3=k-2\ep-\de+1$, $b_4=-n+k-\ep$, $c_4=1$,
$N=n-k$, and then let $\de$ tend to $0$. This leads to the identity
\begin{multline} \label{eq:S1}
\sum _{j=k} ^{n}
\left(\frac n2-j+\ep\right)
\left(\frac {n!} {(1-\ep)_{j}\,(1+\ep)_{n-j}}\right)^3
\binom {n+j-\frac 12-\ep} {n}
           \binom {2n-j-\frac {1} {2}+\ep} {n}\\
=-\frac {1} {2}\left(k-\ep-\frac {1} {2}\right)
\sum _{0\le i_1\le i_2\le i_3\le n-k} ^{}
(-1)^{i_2} \frac {i_3!} {i_1!\,(i_2-i_1)!\,(i_3-i_2)!}
\frac {(\frac 12-\ep)_{n}} {(\frac 12-\ep)_{k}\,(1+\ep)_{n-k}}\\
\cdot
           \frac {(n-\ep+\frac 12)_{k+i_1}} {(1-\ep)_{k+i_1}}
           \frac {(n+i_1-i_2+\ep+\frac 12)_{n-k-i_1}} {(1+\ep)_{n-k-i_1}}
           \frac {n!} {(1-\ep)_{k+i_2} \,(1+\ep)_{n-k-i_2}}\\
\cdot
\frac {(\frac 12+n+i_1-i_2)_{i_2-i_1}} {(\frac 12+n+\ep+i_1-i_2)_{i_2-i_1}}
\frac {(n-\frac 12-i_3+\ep)_{i_3+1}} {(n-\frac 12-i_3-\ep)_{i_3+1}}\\
\cdot
           \frac {(\ep)_{i_3-i_2}\,(1-2\ep)_{k+i_2}\,
           (\frac 12+\ep)_{n-i_3-1}}
           {(1-2\ep)_{k-1}\,(\frac 12+\ep)_{n-k-i_3}\,
           (1-\ep)_{k+i_3}}.
\end{multline}
Using the notations $R_1(\al,\be;t)$ and $R_2(\al,\be;t)$ 
for elementary bricks that were
introduced in Lemmas~\ref{lem:R1} and \ref{lem:R2}, and the notations
\begin{align*}
R_3(n,i_1,i_2,\ep)&=
\frac {(\frac 12+n+i_1-i_2)_{i_2-i_1}} {(\frac
12+n+\ep+i_1-i_2)_{i_2-i_1}},\\
R_4(n,i_3,\ep)&=
\frac {(n-\frac 12-i_3+\ep)_{i_3+1}} {(n-\frac
12-i_3-\ep)_{i_3+1}},\\
R_5(n,k,i_2,i_3,\ep)&=2^{2(k-1)}
           \frac {(\ep)_{i_3-i_2}\,(1-2\ep)_{k+i_2}\,
           (\frac 12+\ep)_{n-i_3-1}}
           {(1-2\ep)_{k-1}\,(\frac 12+\ep)_{n-k-i_3}\,
           (1-\ep)_{k+i_3}}
\end{align*}
for the {\it special bricks} $R_3(n,i_1,i_2,\ep)$, $R_4(n,i_3,\ep)$,
and $R_5(n,k,i_2,i_3,\ep)$,
use of \eqref{eq:S1} in \eqref{eq:bn2} yields
\begin{multline} \label{eq:S2}
2^{4n}\dd_{2n}^2{\mathbf b}_n
=-(-1)^n\dd_{2n}^2\sum _{e=1} ^{3}\sum _{k=1} ^{n}\frac {(-1)^{k}}
{\(k-\frac {1} {2}\)^e}
\frac {1} {(3-e)!}\frac {\partial^{3-e}} {\partial \ep^{3-e}}
\Bigg(\left(2k-2\ep-1\right)\\
\times
\sum _{0\le i_1\le i_2\le i_3\le n-k} ^{}
(-1)^{i_2} \frac {i_3!} {i_1!\,(i_2-i_1)!\,(i_3-i_2)!}
\cdot
R_2(n,k;1-\ep)\cdot\ep\cdot R_1(0,n+1-k;\ep)\\
\cdot
R_2(k+i_1,0;n-\ep+1)\cdot(-\ep)\cdot R_1(0,k+i_1;-\ep)\\
\cdot
R_2(n-k-i_1,0;n+i_1-i_2+1)\cdot \ep\cdot R_1(0,n-k-i_1;\ep)\\
\cdot
(-1)^{k+i_2}\ep\cdot R_1(-k-i_2,n-k-i_2+1;\ep)
\cdot
R_3(n,i_1,i_2,\ep)\cdot R_4(n,i_3,\ep)\cdot
R_5(n,k,i_2,i_3,\ep)\Bigg)\Bigg\vert_{\ep=0}.
\end{multline}
We can rewrite this in the form
\begin{multline*}
2^{4n}\dd_{2n}^2{\mathbf b}_n
=-(-1)^n\dd_{2n}^2\sum _{e=1} ^{3}\sum _{k=1} ^{n}\frac {(-1)^{k}}
{\(k-\frac {1} {2}\)^e}
\frac {1} {(3-e)!}\frac {\partial^{3-e}} {\partial \ep^{3-e}}
\Bigg\{\left(2k-2\ep-1\right)\\
\cdot
\sum _{0\le i_1\le i_2\le i_3\le n-k} ^{}
C(i_1,i_2,i_{3})\cdot
R_5(n,k,i_{2},i_{3};\ep)
\prod _{h=1} ^{M}t_h(n,k,i_1,i_2,i_{3};\ep)
\Bigg\}\Bigg\vert_{\ep=0},
\end{multline*}
where each $C(i_1,i_2,i_{3})$ is an integer and
each $t_h$ is an expression $R_1(\al,\be;\pm \ep+K)$
with $\al\ge \be$, an expression $R_1(\al,\be;\pm \ep)$ multiplied
by $\pm\ep$ with $\al< \be $, an expression $R_2(\al,\be;\pm \ep+K)$
with $\al\ge \be$, or one of $R_3(n,i_1,i_2,\ep)$ and
$R_4(n,i_3,\ep)$.

{
\allowdisplaybreaks
By Leibniz's formula, this last expression can be expanded into 
\begin{align} \notag
2^{4n}\dd_{2n}^2{\mathbf b}_n
&=-(-1)^n\sum _{e=1} ^{3}\sum _{k=1} ^{n}\frac {2(-1)^{k}\dd_{2n}^{e-1}}
{\(k-\frac {1} {2}\)^{e-1}}
\notag\\
\notag
&\kern1cm
\cdot\dd_{2n}^{3-e}\Bigg\{
\sum _{\ell_0+\dots+\ell_M=3-e} ^{}\frac {1}
{\ell_0!\,\ell_1!\cdots\ell_M!}
\sum _{0\le i_1\le i_2\le i_{3}\le n-k} ^{}
C(i_1,i_2,i_{3})\\
\notag
&\kern3cm
\cdot
\frac {\partial^{\ell_0}} {\partial\ep^{\ell_0}}
R_5(n,k,i_{2},i_{3};\ep)
\prod _{h=1} ^{M}\frac {\partial^{\ell_h}} {\partial\ep^{\ell_h}}
t_h(n,k,i_1,i_2,i_{3};\ep)
\Bigg\}\Bigg\vert_{\ep=0}\\
\notag
&+(-1)^n\sum _{e=1} ^{3}\sum _{k=1} ^{n}\frac {2(-1)^{k}\dd_{2n}^{e}}
{\(k-\frac {1} {2}\)^{e}}\\
&\kern1cm
\cdot
\dd_{2n}^{2-e}\Bigg\{
\sum _{\ell_0+\dots+\ell_M=2-e} ^{}\frac {1}
{\ell_0!\,\ell_1!\cdots\ell_M!}
\sum _{0\le i_1\le\dots\le i_{3}\le n-k} ^{}
C_2(i_1,i_2,i_{3})
\notag\\
&\kern3cm
\cdot
\frac {\partial^{\ell_0}} {\partial\ep^{\ell_0}}
R_5(n,k,i_{2},i_{3};\ep)
\prod _{h=1} ^{M}\frac {\partial^{\ell_h}} {\partial\ep^{\ell_h}}
t_h(n,k,i_1,i_2,i_{3};\ep)
\Bigg\}\Bigg\vert_{\ep=0}.
\label{eq:final1}
\end{align}}

Now, for any $h$ with $1\le h\le M$, we claim that
$$\frac {\dd_{2n}^{\ell_h}} {\ell_h!}\frac {\partial^{\ell_h}} {\partial\ep^{\ell_h}}
t_h(n,k,i_1,i_2,i_{3};\ep)\Big\vert_{\ep=0}
$$
is an integer. Indeed, if $t_h(n,k,i_1,i_2,i_{3};\ep)$ is one of the
elementary bricks $R_1(\dots)$ (possibly multiplied by $\pm\ep$) or
$R_2(\dots)$, then this follows directly from Lemmas~\ref{lem:R1} and
\ref{lem:R2}. If $t_h(n,k,i_1,i_2,i_{3};\ep)$ is one of the special
bricks $R_3(\dots)$ or $R_4(\dots)$, this can be seen directly.
Since $2(k-\frac {1} {2})$ divides $\dd_{2n}$, 
Identity~\eqref{eq:final1} would imply the assertion of the theorem
once we could prove that
\begin{equation} \label{eq:R5}
\frac {\dd_{2n}^{\ell_0}} {\ell_0!}\frac {\partial^{\ell_0}} {\partial\ep^{\ell_0}}
R_5(n,k,i_{2},i_{3};\ep)\Big\vert_{\ep=0}
\end{equation}
is an integer as well.

To accomplish this, we distinguish between two cases.
If $i_{2}=i_{3}$, then $R_5(n,k,i_{2},i_{3};\ep)$ can be factored
as follows:
\begin{align*}
R_5(n,k,i_{2},i_{3};\ep)
&=R_5(n,k,i_{3},i_{3};\ep)\\
&=2^{2(k-1)}
\frac {(1-2\ep)_{k+i_{3}}
\,(\frac {1} {2}+\ep)_{n-i_{3}-1}}
{(1-2\ep)_{k-1}\,(1-\ep)_{k+i_{3}}\,
(\frac {1} {2}+\ep)_{n-k-i_{3}}}\\
&= R_1(k+i_{3},0;1-2\ep)\\
&\kern1cm
\cdot(-\ep)\cdot R_1(0,k+i_{3}+1;-\ep)
\cdot R_2(k-1,0;n-k-i_3+\ep)\\
&\kern1cm
\cdot (-2\ep)\cdot R_1(0,k;-2\ep).
\end{align*}
Another application of Leibniz's formula and of
Lemmas~\ref{lem:R1} and \ref{lem:R2} show that \eqref{eq:R5} is an
integer for $i_{2}=i_{3}$.

If $i_{2}<i_{3}$, one observes that
$$R_5(n,k,i_{2},i_{3};\ep)=\ep\cdot
R_6(n,k,i_{2},i_{3};\ep),$$
where $R_6(\dots)$ is the special brick 
defined in Lemma~\ref{lem:briques1}.
Consequently, for $\ell_0\ge1$, we have
\begin{equation*}
\frac {1} {\ell_0!}\frac {\partial^{\ell_0}} {\partial\ep^{\ell_0}}
R_5(n,k,i_{2},i_{3};\ep)\Bigg\vert_{\ep=0}
=\frac {1} {(\ell_0-1)!}\frac {\partial^{\ell_0-1}} {\partial\ep^{\ell_0-1}}
R_6(n,k,i_{2},i_{3};\ep)\Bigg\vert_{\ep=0}.
\end{equation*}
The above relation together with Lemma~\ref{lem:briques1} with
$m_1=i_{3}$ and $m_2=i_{2}$ then shows that \eqref{eq:R5} is also an
integer for $i_{2}<i_{3}$.

This completes the proof of the theorem.
\end{proof}

\begin{lem} \label{lem:briques1}
Let
$$R_6(n,k,m_1,m_2;\ep)=2^{2(k-1)}\frac {(1+\ep)_{m_1-m_2-1}\,(1-2\ep)_{k+m_{2}}
\,(\frac {1} {2}+\ep)_{n-m_{1}-1}}
{(1-2\ep)_{k-1}\,(1-\ep)_{k+m_{1}}\,
(\frac {1} {2}+\ep)_{n-k-m_{1}}}.$$
Then, for all integers $n,k,m_1,m_2,H$ with $H\ge0$ and $0\le m_2<m_1\le n-k$, 
the number
\begin{equation} \label{eq:H0}
\textup{d}_{2n}^{H+1}\cdot \frac {1} {H!}\,
\frac {\partial ^H} {\partial \ep^H}
R_6(n,k,m_1,m_2;\ep)\Big\vert_{\ep=0}
\end{equation}
is an integer.
\end{lem}

\begin{proof}
We loosely follow analogous arguments in the proof of 
\cite[Lemme~11]{KrRiAA}.
In fact, the arguments given in the last paragraph here show that that
proof could have been simplified.

We shall show that, for all integers $1\le f_1\le
f_2\le\dots\le f_{H+1}\le 2n$, the number
\begin{equation} \label{eq:H2}
\textup{d}_{2n}^{H+1}\cdot \frac {1} {H!}\,
2^{2(k-1)}\frac {(m_1-m_2-1)!\,(k+m_{2})!
\,(\frac {1} {2})_{n-m_{1}-1}}
{(k-1)!\,(k+m_{1})!\,
(\frac {1} {2})_{n-k-m_{1}}}
\frac {1} {f_1f_2\cdots f_{H}}
\end{equation}
is an integer. In view of the definition of $R_6(n,k,m_1,m_2;\ep)$,
this implies that \eqref{eq:H0} is an integer.

We prove the above claim by verifying that the $p$-adic valuation of
\eqref{eq:H2} is non-negative for all prime numbers $p$.
Writing $[\al]$ for the greatest integer less than or equal to $\al$,
this $p$-adic valuation is equal to
\begin{multline} \label{eq:H3}
(H+1)\cdot[\log_p(2n)]+
\sum _{\ell=1} ^{\infty}\bigg(
\left[\frac {k+m_2} {p^\ell}\right]
+\left[\frac {m_1-m_2-1} {p^\ell}\right]
+\left[\frac {2n-2m_1-2} {p^\ell}\right]
-\left[\frac {n-m_1-1} {p^\ell}\right]
\\
-\left[\frac {k-1} {p^\ell}\right]
-\left[\frac {k+m_1} {p^\ell}\right]
-\left[\frac {2n-2k-2m_1} {p^\ell}\right]
+\left[\frac {n-k-m_1} {p^\ell}\right]
\bigg)
-
\sum _{h=1} ^{H}v_p(f_h)
\end{multline}
for {\it any} prime number $p$ (also for $p=2$!).
If $p>2n$, it is obvious that this expression is non-negative since all
terms vanish. Hence, from now on we assume that $p\le 2n$.

In fact, the conditions on $k,n,m_1,m_2$ imply that the terms of the
infinite series in \eqref{eq:H3} vanish for $\ell>[\log_p(2n)]$. 
The expression \eqref{eq:H3} can therefore be rewritten in the form
\begin{multline} \label{eq:H4}
[\log_p(2n)]+\sum _{\ell=1} ^{[\log_p(2n)]}\bigg(
\left[\frac {k+m_2} {p^\ell}\right]
+\left[\frac {m_1-m_2-1} {p^\ell}\right]
+\left[\frac {2n-2m_1-2} {p^\ell}\right]
-\left[\frac {n-m_1-1} {p^\ell}\right]
\\
-\left[\frac {k-1} {p^\ell}\right]
-\left[\frac {k+m_1} {p^\ell}\right]
-\left[\frac {2n-2k-2m_1} {p^\ell}\right]
+\left[\frac {n-k-m_1} {p^\ell}\right]
\bigg)
-
\sum _{h=1} ^{H}\(v_p(f_h)-[\log_p(2n)]\).
\end{multline}
Since, by definition, $1\le f_h\le 2n$ for all $h$, the terms in the
sum over $h$ are non-positive. Hence, it suffices to show that the
summands in the sum over $\ell$ are all at least $-1$.

In order to accomplish this, we write $N=\{n/p^\ell\}$, $K=\{k/p^\ell\}$,
$M_1=\{m_{1}/p^\ell\}$, $M_2=\{m_{2}/p^\ell\}$ for the fractional
parts of $n/p^\ell$, $k/p^\ell$, $m_{1}/p^\ell$ and $m_{2}/p^\ell$,
respectively.
With these notations, the summand of the sum over $\ell$ becomes
\begin{multline} \label{eq:H5}
\left[ K+M_2\right]
+\left[ M_1-M_2-\frac {1} {p^\ell}\right]
+\(\left[2N-2M_1-\frac {2} {p^\ell}\right]
-\left[N-M_1-\frac {1} {p^\ell}\right]\)
\\
-\left[K-\frac {1} {p^\ell}\right]
-\left[K+M_1\right]
-\big(\left[2N-2K-2M_1\right]
-\left[N-K-M_1\right]\big).
\end{multline}

We first discuss the case $K=0$. For this special choice of $K$,
the expression in \eqref{eq:H5} reduces to
\begin{multline} \label{eq:H5a}
\left[ M_1-M_2-\frac {1} {p^\ell}\right]
+\(\left[2N-2M_1-\frac {2} {p^\ell}\right]
-\left[N-M_1-\frac {1} {p^\ell}\right]\)
\\
+1-\big(\left[2N-2M_1\right]
-\left[N-M_1\right]\big).
\end{multline}
Since, by elementary properties of the (weakly) increasing
function $x\mapsto [2x]-[x]$, we have
$$
\(\left[2N-2M_1-\frac {2} {p^\ell}\right]
-\left[N-M_1-\frac {1} {p^\ell}\right]\)
-\big(\left[2N-2M_1\right]
-\left[N-M_1\right]\big)\ge-1,
$$
the expression in \eqref{eq:H5a} is indeed $\ge-1$.

From now on let $K>0$, i.e., $K\ge \frac {1} {p^\ell}$.
In this case, clearly, $\left[K-\frac {1} {p^\ell}\right]=0$ and 
$$
\(\left[2N-2M_1-\frac {2} {p^\ell}\right]
-\left[N-M_1-\frac {1} {p^\ell}\right]\)
-\big(\left[2N-2K-2M_1\right]
-\left[N-K-M_1\right]\big)\ge0.
$$
Hence, if the expression in \eqref{eq:H5} wants to be $\le-2$, then
we must have $[K+M_2]=0$, $\left[M_1-M_2-\frac {1} {p^\ell}\right]=-1$ and
$[K+M_1]=1$, that is
\begin{align} \label{eq:H6a}
K+M_2&<1,\\
M_1-M_2-\frac {1} {p^\ell}&<0,\label{eq:H6b}\\
K+M_1&\ge1.\label{eq:H6c}
\end{align}
But a combination of \eqref{eq:H6a} and \eqref{eq:H6c} yields
$M_1-M_2>0$, which contradicts \eqref{eq:H6b} since the denominators
of the rational numbers $M_1$ and $M_2$ are both $p^\ell$.
\end{proof}

\end{document}